\documentclass[11pt,reqno]{amsart}
\usepackage{amssymb,latexsym}
\usepackage{graphicx}
\usepackage{fancyhdr,amssymb}
\usepackage{url}        %does nice formatting of URLs

\theoremstyle{plain} \numberwithin{equation}{section}
\newtheorem{thm}{Theorem}[section]
\newtheorem{theorem}[thm]{Theorem}
\newtheorem{lemma}[thm]{Lemma}
\newtheorem{corollary}[thm]{Corollary}

\begin{document}

\fancyhead{}
\renewcommand{\headrulewidth}{0pt}

\setcounter{page}{1}

\title[Arithmetic Functions of Balancing Numbers]{Arithmetic Functions of Balancing Numbers}
\author{Manasi Kumari Sahukar}
\address{Department of Mathematics\\
                National Institute of Technology\\
                Rourkela, Odisha\\
                India}
\email{manasi.sahukar@gmail.com}

\author{G.K.Panda}
\address{Department of Mathematics\\
               National Institute of Technology\\
                Rourkela, Odisha\\
                India}
\email{gkpanda\_nit@rediffmail.com}

\begin{abstract}
Two inequalities involving the Euler totient function and the sum of the $k$-th powers of the divisors of balancing numbers are explored.
\end{abstract}

\maketitle

\maketitle 
\textbf{2010 Subject classification [A.M.S.]:} 11B39, 11A25.

\section{introduction}
For any positive integer $n$, the Euler totient function $\phi(n)$ is defined as number of positive integers less than $n$ and relatively prime to $n$, and $\sigma_{k}(n)$ denote the sum of the $k$-th power of divisors of $n$. If $k=0$, $\sigma_{k}(n)$ reduces to the function $\tau(n)$, which counts the number of positive divisors of $n$. For many centuries, mathematicians were more concerned on the arithmetic functions of natural numbers and solved many Diophantine equations concerning these functions. Subsequently, some researchers focus their attention on study of arithmetic functions relating to binary recurrence sequences such as Fibonacci sequence, Lucas sequence, Pell sequence and associated Pell sequence.

In 1997, Luca \cite{Luca 1997} showed that the Euler totient function for the homogeneous binary recurrence sequences $\{u_n\}_{n\geq0}$ satisfy the inequality $\phi(|u_n|) \geq |u_{\phi(n)}|$ for those binary recurrences with characteristic equations having real roots and the inequality is not valid for those recurrences with characteristic equations having complex roots. In \cite{Luca 1998}, he  proved that the $n$-th Fibonacci number satisfies $\sigma_k(F_n) \leq F_{\sigma_k(n)}$ and $\tau(F_n) \geq F_{\tau(n)}$ for all $n \geq 1$. Motivated by these works, we study two similar inequalities involving arithmetic functions of balancing numbers.

Recall that a natural number $B$ is a balancing number with balancer $R$ if the pair $(B,R)$ satisfies the Diophantine equation $1+2+\cdots+(B-1)=(B+1)+\cdots+(B+R)$. If $B$ is a balancing number then $8B^2+1$ is a perfect square and its positive square root is called a Lucas balancing number. The $n$-th balancing number is denoted by $B_n$ while the $n$-th Lucas-balancing number is denoted by $C_n$. The balancing numbers satisfy the binary recurrence $B_{n+1}=6B_n-B_{n-1}$, $B_0=0$, $B_1=1$ which holds for $n\geq1$, while the Lucas-balancing numbers satisfy a binary recurrence identical with that of balancing numbers, however with initial values $C_0=1$, $C_1=3$. The characteristic equation of these recurrences is given by $x^2-6x+1=0$ whose roots are $\alpha=3+2\sqrt{2}$ and $\beta=3-2\sqrt{2}$. The Binet forms of balancing and Lucas-balancing numbers are given by 
$$B_n=\frac{\alpha^n-\beta^n}{\alpha-\beta}, C_n=\frac{\alpha^n+\beta^n}{2}$$ (see \cite{Behera 1999,Ray 2009}).

Given a natural number $A>2$, the sequence arising out of the class of binary recurrence $x_{n+1}=Ax_n-x_{n-1}$ with initial terms  $x_0=0$, $x_1=1$ is known as a balancing-like sequence because the case $A=6$ corresponds to balancing sequence \cite{Rout 2012}. It is interesting to note that when $A=2$, the above recurrence relation generates the sequence of natural numbers. Further, when $A=3$, the corresponding balancing-like sequence coincides with the sequence of even indexed Fibonacci numbers. The balancing-like sequences (and hence the balancing sequence) satisfy certain identities in which they behave like natural numbers \cite{Panda 2009, Rout 2012} and hence these sequences are considered as generalization of the sequence of natural numbers.

\section{Auxiliary results}

To establish the inequalities concerning arithmetic functions of balancing numbers, we need the following results. Some results of this section are new and hence we provide proofs of such results.

The following lemma presents some basic properties of balancing numbers. 

%%%% Lemma 2.1

%%%%%%%%%%%%%%%
\begin{lemma}\label{2.1}{\rm{(\cite{Panda 2009}, Theorem 2.5, \cite{Ray 2009}, Theorem 5.2.6)}}
If $m$ and $n$ are natural numbers then 
\begin{enumerate}
\item $B_{m+n}=B_mC_{n}+C_{m}B_n$.\\
\item $5^{n-1}<B_n<6^{n-1}$ for $n\geq3$.
\end{enumerate}
\end{lemma}

%%%% Lemma 2.2

The following two lemmas deal with the divisibility property of balancing numbers.
\begin{lemma}\label{2.2}{\rm{(\cite{Panda 2009}, Theorem 2.8)}}
If $m$ and $n$ are natural numbers then $B_m$ divides $B_n$ if and only if $m$ divides $n$.
\end{lemma}

%%%% Lemma 2.3

\begin{lemma}\label{2.3}{\rm{(\cite{Panda 2009}, Theorem 2.13)}}
If $m$ and $n$ are natural numbers then $(B_m,B_n)=B_{(m,n)}$, where $(x,y)$ denotes the greatest common divisor of $x$ and $y$.
\end{lemma}

Given any two nonzero integers $A$ and $B$, we consider the second order linear recurrence sequence $\{w_n\}_{n\geq0}$  defined by $w_{n+1}=Aw_n+Bw_{n-1}$ with initial terms $w_0=0$ and $w_1=1$. If $A^2+4B>0$ then the characteristic equation $x^2-Ax-B=0$ has distinct real  roots $\alpha=\frac{A+\sqrt{A^2+4B}}{2}$, $\beta=\frac{A-\sqrt{A^2+4B}}{2}$ and the Binet form is given by $w_n=\frac{\alpha^n-\beta^n}{\alpha-\beta}$. A prime $p$ is called as primitive divisor of $w_n$ if $p$ divides $w_n$ but does not divide $w_m$ for $0<m<n$. 

The following two lemmas deal with the existence of primitive divisors of the sequence $\{w_n\}_{n\geq0}$ described in the last paragraph and the balancing sequence $\{B_n\}_{n\geq0}$.

%%%% Lemma 2.4

\begin{lemma}\label{2.4}{\rm{(\cite{Yabuta 2001}, Theorem 1)}}.
If the roots $\alpha$ and $\beta$ are real and $n\neq 1,2,6,12,$ then $w_n$ contains at least one primitive divisor.
 \end{lemma}
 
 %%%% Lemma 2.5
 
 \begin{lemma}\label{2.5}
 A primitive prime factor of $B_n$ exists if $n>1$.
 \end{lemma}
 \begin{proof}
 In Section 1, we have seen that the characteristic roots $\alpha=3+2\sqrt{2}$ and $\beta=3-2\sqrt{2}$ corresponding to the binary recurrence of the balancing sequence are real. Hence, by virtue of Lemma \ref{2.4}, $B_n$ has a primitive divisor for all $n\in\mathbb{Z}$ except possibly $n \in\{1,2,6,12\}$. But one can easily check that $B_2=6$, $B_6=6930$ and $B_{12}=271669860$ have primitive divisors 3, 11 and 1153 respectively. 
  \end{proof}
 
 %%%% Lemma 2.6
\begin{lemma}\label{2.6}{\rm{(\cite{Rout 2014}, Theorem 3.2)}}
If $p$ is a prime of the form $8x\pm1$ then $p$ divides $B_{p-1}$, further if the prime $p$ is of the form $8x\pm3$ then $p$ divides $B_{p+1}$.
\end{lemma}
The following lemma provides bounds for ratios of two consecutive balancing numbers.

%%%% Lemma 2.7

\begin{lemma}\label{2.7}
For any natural number $n$,  $\frac{B_{n+1}}{B_n}>\alpha$. 
\end{lemma}
\begin{proof}
Using the fact that $\alpha \beta=1$, we get
\begin{align*}
B_{n+1}-\alpha B_n &=\frac{\alpha^{n+1}-\beta^{n+1}}{\alpha-\beta}-\alpha \frac{\alpha^{n}-\beta^{n}}{\alpha-\beta}\\
&= \frac{\beta^{n-1}-\beta^{n+1}}{\alpha-\beta}\\
 &= \frac{\beta^{n-1}(1-\beta^2)}{\frac{1}{\beta}-\beta}\\
&=\beta^n >0
\end{align*}
\end{proof}

The following corollary is a direct consequence of Lemma \ref{2.7}.
%%%% Lemma 2.8
\begin{corollary} \label{2.8}
For all natural number $n\geq2$, $B_n>\alpha^{n-1}$.
\end{corollary}

The following Lemma provides an upper bound for the $n$-th balancing number. 
%%%% Lemma 2.9
\begin{lemma}\label{2.9}
For all natural number $n\geq1$, $B_n<\alpha^n$.
\end{lemma}
\begin{proof}
It follows from the Binet formula for balancing numbers that $B_n=\frac{\alpha^{n}-\beta^{n}}{\alpha-\beta}<\frac{\alpha^n}{4\sqrt{2}}<\alpha^n$.
\end{proof}
The following Lemma gives a comparison of the $(m+n)$-th and $(m-n)$-th balancing numbers with the product and ratio of the $m$-th and $n$-th balancing numbers, respectively.

%Lemma 2.10

\begin{lemma}\label{2.10}
If $m$ and $n$ are two natural numbers, then $B_{m+n}>B_mB_n$ and $B_{m-n}< \frac{B_m}{B_n}$.
\end{lemma}
\begin{proof}
Let $m$ and $n$ be natural numbers. Since by Lemma \ref{2.1}, $B_{m+n}=B_mC_n+C_mB_n$  and from the definition of Lucas-balancing numbers $B_n<C_n$, it follows that $B_{m+n}>B_mB_n$. Since $B_m=B_{(m-n)+n}> B_{m-n}B_n$, the inequality $B_{m-n}< \frac{B_m}{B_n}$ follows.
\end{proof}

The next lemma gives a comparison of $n^k$-th balancing number with the $k$-th power of $n$-th balancing number.

%%%% Lemma 2.11

\begin{lemma}\label{2.11}
$B_{n^k} > {B_n}^k ~~~ for ~~n \geq 2~and~k \geq 1.$
\end{lemma}
\begin{proof}

Let $m$,$n$ and $k$ be natural numbers. Since $B_m\geq B_n$ whenever $m\geq n$, and $n^k \geq nk$, for all $n \geq 2$, it follows that $B_{n^k} \geq B_{nk}$. Now, using Lemma \ref{2.10} and simple mathematical induction, it is easy to see that $B_{nk} > {B_n}^k$.
\end{proof}

The following lemma gives certain bounds involving the arithmetic functions. For the proof of this lemma the readers are advised to go through \cite{Luca 1997} and  \cite{Ward 1955}. 

%%%% Lemma 2.12

\begin{lemma}\label{2.12}{\rm{(\cite{Luca 1997}, Lemma 3)}}
Let $m$ and $n$ be natural numbers.
\begin{enumerate}
\item If $n \geq 2\cdot 10^9$, then $\phi(n) >\frac{n}{\rm{log}~n}$.\\
\item If $1 \leq n < 2\cdot 10^9$, then $\phi(n)>\frac{n}{6}$.\\
\item If $m \geq 2$ and $k \geq 1$, then $\frac{m}{\phi(m)}> \frac{\sigma_k(m)}{m^k}$.\\
\item If $n$ is not prime, then $n-\phi(n) \geq \sqrt{n}$.\\
\item If $n$ is not prime, then $\sigma_k(n)-n^k \geq \sqrt{n^k}$.
\end{enumerate}
\end{lemma}

The following lemma deals with an inequality involving the Euler totient function of balancing numbers.

%Lemma 2.13

\begin{lemma}\label{2.13}

For any natural numbers $n$, $\phi(B_n) \geq B_{\phi(n)}$ and equality holds only if $n=1$.
\end{lemma}
\begin{proof}
Consider the  binary recurrence sequence $\{w_n\}_{n\geq1}$  defined just after Lemma \ref{2.3}. Luca \cite{Luca 1997} proved that if the characteristic roots $\alpha$ and $\beta$ are real then $\phi(|w_n|) \geq |w_{\phi(n)}|$. Since the characteristic roots $\alpha=3+2\sqrt{2}$ and $\beta=3-2\sqrt{2}$  corresponding to the recurrence relation of the balancing sequence are real, the inequality $\phi(B_n) \geq B_{\phi(n)}$ holds for all $n \geq 1$.
\end{proof}
The following lemma will play a very crucial role while proving an important result of this paper.

%%%% Lemma 2.14

\begin{lemma}\label{2.14}
If $n$ is an odd prime and $B_n=p_1^{\gamma_1}\cdots p_t^{\gamma_t}$ is the canonical decomposition of $B_n$ then $p_i \geq 2n-1$ for $i=1,\ldots, t$. Further, if  the inequality $\sigma_k(B_n)>B_{\sigma_k(n)}$ is satisfied for all natural numbers $k$ and $n \geq 2$, then $t> 2(n-1)~ {\rm{log}}~5$.

\end{lemma}
\begin{proof}
Let $p$ be any odd prime. By virtue of Lemma \ref{2.6}, $p|B_{p+1}$ or $p|B_{p-1}$. Since $p+1$ and $p-1$ both divide $p^2-1$, it follows from Lemma \ref{2.2} that both $B_{p-1}$ and $B_{p+1}$ divide $B_{p^2-1}$ and hence $p|B_{p^2-1}$. If $p$ is one of the primes $p_1, p_2,\ldots, p_t$ then $p|B_n$ and hence $p |(B_{p^2-1},B_n)$. Since $(B_{p^2-1},B_n)= B_{(p^2-1,n)}$, by virtue of Lemma \ref{2.3}, it follows that $p|B_{(p^2-1,n)}$.
 If $n \nmid p^2-1$, then $(p^2-1,n)=1$ and then $p|B_1=1$ which is not possible. Thus, $n|p^2-1$ and since $n$ is a prime, $n|p+1$ or $n|p-1$ and hence  $p \equiv \pm1(\bmod~ n)$. Clearly $p \neq n \pm 1$ since $p$ and $n$ are both primes and $n>2$. Hence $p \geq 2n-1$. This proves the first part.\\

We next prove the second part assuming that the inequality $\sigma_k(B_n)\geq B_{\sigma_k(n)}$ holds for all natural numbers $k$ and $n\geq 2$. Since
\begin{equation*}
\frac{B_n}{\phi(B_n)}=\frac{B_n}{B_n \prod_{i=1}^t \big(1-\frac{1}{p_i}\big) }=\prod_{i=1}^t\bigg(1+\frac{1}{p_i-1}\bigg),
\end{equation*}
using Lemma \ref{2.7}, Lemma \ref{2.11} and Lemma \ref{2.12}, we get
\begin{equation}\label{2.1}
\prod_{i=1}^t\bigg(1+\frac{1}{p_i-1}\bigg)= \frac{B_n}{\phi(B_n)} > \frac{\sigma_k(B_n)}{{B_n}^k} > \frac{B_{\sigma_k(n)}}{{B_n}^k}>\frac{B_{1+n^k}}{{B_n}^k} \geq 
\frac{B_{1+n^k}}{B_{n^k}} > \alpha >5
\end{equation}
Taking logarithm on both sides, we get
$$\sum_{i=1}^t \text{log}\bigg(1+\frac{1}{p_i-1}\bigg) > \text{log}~5.$$
Since $\text{log}(1+x)< x$ for all $x>0$, we conclude that
\begin{equation}\label{2.2}
\sum_{i=1}^t \frac{1}{p_i-1} >\text{log}~5
\end{equation}
In view of first part of the lemma, it follows that
\begin{equation}\label{2.3}
 \frac{t}{2(n-1)} > \text{log}~5
\end{equation}
which is equivalent to $t> 2(n-1)~ {\rm{log}}~5$.
\end{proof}

\section{MAIN RESULTS}

In this section, we provide  two important theorems dealing with arithmetic functions of the balancing sequence.
In the first theorem, we establish an inequality concerning the sum of $k$-th powers of divisors of balancing numbers.
%theorem 1
%%%%%%
\begin{theorem}\label{3.1}
The balancing numbers satisfy $\sigma_k(B_n)\leq B_{\sigma_k(n)}$ for all $n \geq 1$. Equality holds only if $n=1$.
\end{theorem}
\begin{proof}
Since for $k \geq 1$, $\sigma_k(B_1)=\sigma_k(1)= B_{\sigma_k(1)}$, the assertion of the theorem holds for $n=1$ and all $k\geq1$. For $n\geq 2$, assume to the contrary that
\begin{equation}\label{3.1}
\sigma_k(B_n)>B_{\sigma_k(n)}
\end{equation}
for some $k \geq 1$ and $n \geq 2$.
%In order to derive a contradiction we proceed in two parts.\\
Firstly, we show that Inequality \eqref{3.1} holds only if $n$ is prime. Assume that Inequality \eqref{3.1} holds for some composite number $n \geq 2$.\\

\textbf{Case 1:} Suppose that $B_n < 2\cdot 10^9$. It is only possible when $n<13$. From Lemmas \ref{2.10}, \ref{2.11}, \ref{2.12} and Inequality \eqref{3.1}, it follows that
\begin{equation}\label{3.2}
B_2= 6 > \frac{B_n}{\phi(B_n)} > \frac{\sigma_k(B_n)}{{B_n}^k} > \frac{B_{\sigma_k(n)}}{{B_{n^k}}} > B_{\sigma_k(n)-n^k}
\end{equation}
which implies that $2 > \sigma_k(n)-n^k$. Since $n$ is not prime, it follows from Lemma \ref{2.12} that 
\begin{equation}\label{3.3}
2> \sqrt{n^k}.
\end{equation}
One can easily check that Inequality \eqref{3.3} does not hold for any composite number $n$.\\

\textbf{Case 2:} Suppose that $B_n \geq 2\cdot 10^9$. Then certainly $n \geq 14$. From Lemmas \ref{2.10}, \ref{2.11}, \ref{2.12} and Inequality \eqref{3.1}, it follows that 
\begin{equation}\label{3.4}
\text{log}~B_n > \frac{B_n}{\phi(B_n)} > \frac{\sigma_k(B_n)}{{B_n}^k}  \geq \frac{B_{\sigma_k(n)}}{{B_{n^k}}}>B_{\sigma_k(n)-n^k}.
\end{equation}
Since $\alpha^n > B_n > \alpha^{n-1}$ by Lemmas \ref{2.8}, \ref{2.9} and $\sigma_k(n)-n^k \geq \sqrt{n^k}$, it follows that 
\begin{equation}\label{3.5}
n \text{log}~\alpha > \text{log}~B_n > B_{\sigma_k(n)-n^k} > B_{\sqrt{n^k}} \geq \alpha^{\sqrt{n^k}-1}.
\end{equation}
Further, since $\sqrt{n^k} \geq n$ for $k \geq 2$, Inequality \eqref{3.5} gives
\begin{equation}\label{3.6} 
n \text{log}~\alpha > \alpha^{n-1},
\end{equation}
which holds only when $n=1$, contradicting $n\geq 14$. Thus, the only possibility left is $k=1$. But $k=1$ implies $$n \text{log}~\alpha > \alpha^{\sqrt{n}-1}$$ 
which is true for $n < 5$, which again contradicts  $n \geq 14$. Hence Inequality \eqref{3.1} doesn't hold for any composite number. Hence $n$ is prime.
%%%%%%

Let $n$ be any odd prime. From  Lemma \ref{2.14}, it follows that
\begin{align*}
n\text{log}~\alpha & > \text{log}~B_n\\ & \geq \sum_{i=1}^t \text{log}~p_i\\ & \geq t~\text{log}(2n-1)\\  &> 2(n-1)\text{log}(2n-1)\text{log}~5 
\end{align*}
Hence,
\begin{equation*}
\frac{n \text{log}~\alpha}{2(n-1)\text{log}(2n-1)}-\text{log} ~5 >0
\end{equation*}
which does not hold for any odd prime $n$. Hence $\sigma_k(B_n)\leq B_{\sigma_k(n)}$ for all natural numbers $k$ and odd primes $n$. For $n=2$, we need to show that $\sigma_k(B_2)= \sigma_k(6)=1+2^k+3^k+6^k\leq B_{1+2^k}$. It is sufficient to prove that $4\cdot 6^k < B_{1+2^k}$. Since $2k+2 \leq 2^k$ for all natural number $k \geq 3$, it follows that
$$4\cdot 6^k = 2^{k+2}3^k < \alpha^{2k+2} \leq \alpha^{2^k} <B_{1+2^k}$$
and for $k=1,2$, one can easily check that $\sigma_k(B_2) < B_{1+2^k}$.  This completes the proof.
\end{proof}
In the following theorem, we present an inequality involving another arithmetic function namely the tau function of balancing numbers. We denote the number of distinct prime divisors of $B_n$ by $\omega(B_n)$.
%%%%%%%%%%%%%%

\begin{theorem}\label{3.2}
For any natural number $n$,  $\tau(B_n) > B_{\big\lfloor{\frac{\tau(n)}{3}}\big\rfloor}$, where $\lfloor\cdot\rfloor$ denote the floor function. 
\end{theorem}
\begin{proof}
Let $n$ be a natural number. By virtue of Lemma \ref{2.2}, corresponding to each divisor $m$ of $n$, there exist a primitive divisor of $B_m$ which divides $B_n$ and hence the number of distinct prime divisors of $B_n$ is at least the total number of divisors of $n$, i.e, $\omega(B_n)\geq \tau(n)$ for $n>1$. For each natural number $n$, it is easy to see that $\tau(n)\geq2^{\omega(n)}$. Thus,
\begin{equation}\label{3.7}
\tau(B_n) \geq 2^{\omega(B_n)} \geq 2^{\tau(n)}.
\end{equation}
Since for each natural number $n$, $B_n \leq 6^{n-1}< 8^{n-1}=2^{3n-3}$, it follows that 
\begin{equation}\label{3.8}
B_{\big\lfloor{\frac{n}{3}}\big \rfloor} <2^{n-3},
\end{equation}
Now, from Inequality \eqref{3.7}, we have 
\begin{equation}\label{3.9}
\tau(B_n) \geq 2^{\tau(n)} > 2^{\tau(n)-3}> B_{\big \lfloor{\frac{\tau(n)}{3}}\big \rfloor}.
\end{equation}
This completes the proof.
\end{proof}

\begin{center}
ACKNOWLEDGMENTS
\end{center}

It is our pleasure to thank the anonymous referee for his comments and suggestion that significantly improved the accuracy and presentation of this paper.

\medskip

\noindent MSC 2010: 11B39.

\end{document}